\documentclass[11pt]{amsproc}
\usepackage[T1]{fontenc}
\usepackage[letterpaper,left=1.40in,right=1.40in,top=1.25in,bottom=1.25in]{geometry}
\usepackage{amsmath,amssymb,mathtools,mathrsfs}
\usepackage{enumitem}
\usepackage{microtype}
\usepackage[hidelinks]{hyperref}
\usepackage[nameinlink,noabbrev,capitalize]{cleveref}
\usepackage[dvipsnames]{xcolor} % Fornisce nomi di colori aggiuntivi tra cui ForestGreen

\usepackage{comment}
\usepackage[normalem]{ulem}
\definecolor{cz}{HTML}{1E824C}
\hypersetup{
  pdftitle={Construction of Regular Fat Linear Sets and Non-Representable q-Matroids},
  pdfauthor={Yufei Qiu, Valentino Smaldore, Zhenpeng Wang, Corrado Zanella,
and Ferdinando Zullo},
  pdfsubject={Finite geometry and regular fat linear sets},
  pdfkeywords={regular fat linear set, scalar stabilizer, finite projective geometry, q-matroid}
}

\numberwithin{equation}{section}
\allowdisplaybreaks

\newtheorem{theorem}{Theorem}[section]
\newtheorem{proposition}[theorem]{Proposition}
\newtheorem{lemma}[theorem]{Lemma}
\newtheorem{corollary}[theorem]{Corollary}
\theoremstyle{definition}
\newtheorem{definition}[theorem]{Definition}
\newtheorem{problem}[theorem]{Problem}
\theoremstyle{remark}
\newtheorem{remark}[theorem]{Remark}

\newcommand{\F}{\mathbb F}
\newcommand{\PG}{\operatorname{PG}}
\newcommand{\GL}{\operatorname{GL}}

\newcommand{\GammaL}{\operatorname{\Gamma L}}

\newcommand{\Aut}{\operatorname{Aut}}
\newcommand{\Span}{\operatorname{span}}

\newcommand{\supp}{\operatorname{supp}}

\newcommand{\Stab}{\operatorname{Stab}}
\newcommand{\qnum}[2]{\left[#1\right]_{#2}}

\newcommand{\card}[1]{\lvert #1\rvert}
\newcommand{\set}[1]{\left\{#1\right\}}
\newcommand{\angles}[1]{\left\langle #1\right\rangle}

\newcommand{\fq}{\mathbb F_q}
\newcommand{\cF}{\mathcal F}

\newcommand{\fqn}{\mathbb F_{q^n}}

\title[Construction of Regular Fat Linear Sets]{Construction of Regular Fat Linear Sets and Non-Representable \texorpdfstring{$q$}{q}-Matroids}

\author[Y. Qiu]{Yufei Qiu}
\address{\textbf{Yufei Qiu}\\
Department of Mathematics, The University of Hong Kong, Hong Kong SAR, China}
\email{u3612772@connect.hku.hk}

\author[V. Smaldore]{Valentino Smaldore}
\address{\textbf{Valentino Smaldore}\\
Dipartimento di Tecnica e Gestione dei Sistemi Industriali, Università di Padova, Via S. Nicola 3, 35100 Vicenza, Italy}
\email{valentino.smaldore@unipd.it}

\author[Z. Wang]{Zhenpeng Wang}
\address{\textbf{Zhenpeng Wang}\\
Department of Mathematics, The University of Hong Kong, Hong Kong SAR, China}
\email{u3011717@connect.hku.hk}

\author[C. Zanella]{Corrado Zanella}
\address{\textbf{Corrado Zanella}\\
Dipartimento di Tecnica e Gestione dei Sistemi Industriali, Università di Padova, Via S. Nicola 3, 35100 Vicenza, Italy}
\email{corrado.zanella@unipd.it}

\author[F. Zullo]{Ferdinando Zullo}
\address{\textbf{Ferdinando Zullo}\\
Dipartimento di Matematica e Fisica, Università degli Studi della Campania ``Luigi Vanvitelli'', Viale Lincoln 5, 81100 Caserta, Italy}
\email{ferdinando.zullo@unicampania.it}

\date{today}


\begin{document}

\begin{abstract}
Let $q$ be a prime power. 
An $(r, i)$-regular fat $\fq$-linear set has exactly $r$ points of weight $i\ge2$, called fat points. 
All of its remaining points have weight one, and at least one such point occurs. 
We solve two open problems posed by three of the authors of this paper.

First, we derive divisibility constraints on the point weights using the scalar stabilizer of the defining subspace. 
Consequently, a regular fat linear set of rank $2i$ in $\PG(1,q^n)$ has at most two or exactly $q + 1$ fat points. 
In the latter case, the fat points form an $\fq$-subline.

Second, we construct higher-dimensional regular fat linear sets whose
fat points do not form an $\F_q$-subgeometry. In particular, for every
$k\geq 3$ we obtain examples of rank $2k$ whose fat locus consists of an
$\F_q$-subgeometry $\PG(k-2,q)$ in $\PG(k-1,q^4)$ together with one external point.
This construction is then generalized further.
Finally, we apply our projective-line results to $q$-matroid representability and present an infinite family of non-representable rank-two $q$-matroids arising from partial spreads.
\end{abstract}

\maketitle

\section{Introduction}

The point-weight distribution is one of the basic vector invariants of an $\F_q$-linear set and has played a central role in several problems in finite geometry and coding theory.  At one extreme, scattered linear sets, whose points all have weight one, have been intensively studied
since the work of Blokhuis and Lavrauw \cite{BL00}, also because of
their connections with rank-metric codes \cite{PZ20,NZ26}.  Allowing a small number of
points of larger weight leads to other well-studied classes.  Linear
sets with a unique point of weight greater than one are the so-called clubs \cite{FS06,LZ15}, while linear sets with two points of complementary weights were studied in \cite{NPSZ22}.  More recently, these configurations have also appeared naturally in questions on $q$-matroid representability:
clubs are related to free products of rank-one uniform $q$-matroids
\cite{ABF25}, whereas linear sets with two high-weight points arise in the study of direct sums of uniform $q$-matroids \cite{AJNZ26}.  

Regular fat linear sets, introduced in \cite{SZZ26} arise as natural extensions of the above linear sets.  An $(r,i)$-regular fat linear set has
exactly $r$ points of a fixed weight $i\geq2$, while all its remaining
points have weight one, with at least one point of weight one occurring.
In other words, these linear sets are scattered apart from a set of $r$ exceptional points.  They simultaneously include scattered linear sets when $r=0$, clubs when $r=1$, and, in the rank-$2i$ case, the two-fat-point configurations when $r=2$, as  considered in
\cite{NPSZ22}.

Let $q$ be a prime power, and let $U$ be an $\F_q$-subspace of $\F_{q^n}^k$.  The associated linear set is
\[
 L_U=\set{\angles{u}_{\fqn}:u\in U\setminus\set{0}}\subseteq \PG(k-1,q^n).
\]
For a projective point $P=\angles{v}_{\F_{q^n}}$, its weight is
\[
 w_U(P)=\dim_{\F_q}(U\cap \angles{v}_{\F_{q^n}}).
\]
Hence, \emph{regular fat linear sets} are linear sets in which the positive point weights take the two values $1$ and $i$, the latter on exactly $r$ points, see \cite{SZZ26}.  
The authors of this paper, in Proposition 6.1, state that, if $r>2$, a $(r,i)$-regular fat linear set of rank $2i$ in $\PG(1,q^n)$ has $q^j+1$ points of weight $i$, $j|n$. The authors, in the final section, ask whether there exist examples of greater cardinality $(j>1)$, and whether higher-dimensional examples can have a number of fat points strictly between the size of a projective basis and that of an $\F_q$-subgeometry.
See Problems~\ref{prob:szz1} and \ref{prob:higher} below.

The first problem turns out to be related to the scalar stabilizer of the defining subspace.
For an $\F_q$-subspace $U\leq \F_{q^n}^e$, $e\in\mathbb Z_{>0}$, define the \emph{scalar stabilizer} of $U$ as
\[
S(U)=\set{a\in \F_{q^n}:aU\subseteq U}.
\]
This is a subfield $\F_{q^j}$ of $\F_{q^n}$. Moreover, every intersection
of $U$ with an $\F_{q^n}$-linear subspace is an $\F_{q^j}$-vector space.
Consequently, $j$ divides every point weight. On the projective line,
the presence of three fat points forces, up to $\GL(2,q^n)$-equivalence,
the product form
$
U=T\times T$.
If $S(T)=\F_{q^j}$, then the fat points form an $\F_{q^j}$-subline.
Since a regular fat linear set has at least one point of weight one,
the divisibility condition forces $j=1$. Hence a rank-$2i$ regular fat
linear set in $\PG(1,q^n)$ has at most two fat points or exactly $q+1$
fat points, and in the latter case they form an $\F_q$-subline.

For the second problem, we give an explicit higher-dimensional
construction.
For a primitive element $\theta\in\F_{q^4}$, let
\[
T=\angles{1,\theta}_{\F_q},
\qquad
R=\angles{1,\theta^2}_{\F_q}.
\]
Spaces $T$ and $R$ have $\F_q$ as a scalar stabilizer and are not scalar-equivalent. Every relevant intersection of scalar translates has a dimension of at most one. Therefore,
\[
U=T^{k-1}\times R\leq \F_{q^4}^{\,k}
\]
defines a regular fat linear set in $\PG(k-1,q^4)$ for every $k \geq 3$. Its fat locus is the disjoint union of an $\F_q$-subgeometry, $\PG(k-2,q)$, contained in a hyperplane and one point outside the hyperplane. Therefore, Problem 2.5 is answered in the affirmative. For $q \geq 3$, combining the same idea with the multi-orbit construction of \cite{Zul23} yields analogous examples with arbitrary fat weight $i \geq 2$.

The above construction is part of a more general block-product
approach. Let $T_1,\ldots,T_c\leq\F_{q^n}$ be pairwise
scalar-inequivalent $i$-dimensional subspaces, occurring with
multiplicities $m_1,\ldots,m_c$. Under suitable one-dimensional
intersection conditions among their scalar translates, the product
\[
U=T_1^{m_1}\times\cdots\times T_c^{m_c}
\]
defines a regular fat linear set. Its fat locus is a disjoint union of
subgeometries, one for each block. Conversely, under the natural
spanning hypothesis, every regular fat defining subspace of rank $ki$
arises from such a block decomposition.
Finally, we apply our projective-line results to $q$-matroid
representability. Starting from an partial $i$-spread
$\mathcal V=\{V_1,\ldots,V_s\}$ of $\F_q^{2i}$, we associate with
$\mathcal V$ a natural loopless rank-two $q$-matroid
$\mathcal M_{\mathcal V}$. We show that, for
$
3\leq s<q^i+1$,
$s\neq q+1$,
these $q$-matroids are not representable over any finite extension of
$\F_q$. 
Additionally, we prove that for $q>2$ and $s=q+1$ nonrepresentable $q$-matroids do exist. 
This provides an infinite family of nonrepresentable rank-two
$q$-matroids. In the particular case $q=2$, $i=2$, and $s=4$, our
construction recovers the $q$-matroid of Ceria and Jurrius
\cite{CJ24}, which they proved to be the smallest nonrepresentable
$q$-matroid.

\bigskip

The paper is organized as follows. In Section~\ref{sec:RFLS} we define regular fat linear sets and describe two problems about them. In Section~\ref{sec:SS} we study scalar stabilizers and their role in the point-weight distribution, and we obtain the classification of rank-$2i$ $(r,i)$-regular fat linear sets on the projective line for $r\ge2$. In Section~\ref{sec:ABP} we introduce admissible block products, and construct regular fat linear sets arising from them (\cref{thm:block-construction}).
Section~\ref{sec:struc} is devoted to the structural constraints satisfied by this construction. In Section~\ref{sec:fam}, we exhibit explicit families of regular fat linear sets in  higher dimension, which, in particular, provide an affirmative answer to the second open problem of \cite{SZZ26}. In Section~\ref{sec:unique} we prove the converse of \cref{thm:block-construction}: under a natural spanning assumption, some regular fat linear sets arise from admissible block products. Their semilinear equivalence and stabilizers are studied in Section~\ref{sec:semil}. Finally, in Section~\ref{sec:q_mat} we apply the results on the projective line to the representability of rank-two $q$-matroids arising from partial spreads, obtaining infinite families of nonrepresentable $q$-matroids.

\section{Regular fat linear sets}\label{sec:RFLS}

Throughout the paper, $q=p^h$ is a prime power. Dimensions without an indicated ground field are taken over the finite field $\F_q$.

We begin with the following definition from \cite{SZZ26}.

\begin{definition}
Let $i\geq 2$ and $r\geq 0$.  The linear set $L_U$ is $(r,i)$-regular fat if the following conditions hold:
\begin{enumerate}[label=\textup{(\roman*)}]
\item exactly $r$ points have weight $i$;
\item every other point of $L_U$ has weight one;
\item at least one point of weight one occurs.
\end{enumerate}
The points of weight $i$ are called \emph{fat}, and we will denote the set of fat points as $\cF(U)$.
\end{definition}

The last condition excludes subfield-linear configurations of constant weight larger than one. 
As can be shown starting from the equations in \cite[Proposition 2.2]{Pol10}, the number of points of an $(r, i)$-regular fat linear set depends only on $r$, $i$, and the rank $\rho$.

\begin{proposition}\label{prop:size}\cite[Proposition 2.3]{SZZ26}
Let $L_U$ be $(r,i)$-regular fat and put $\rho=\dim U$.  Then
\begin{equation}\label{eq:size}
 \card{L_U}=\frac{q^\rho-1-r(q^i-q)}{q-1}.
\end{equation}
\end{proposition}
Weights are unchanged by applying  semilinear isomorphisms on $U$. 
\begin{lemma}\label{lem:projective-invariance}
Let $A\in\GammaL(k,q^n)$ and put $U'=A(U)$, and let $\overline A$ be the projective collineation of $PG(k-1,q^n)$ induced by $A$.
Then
\[
 L_{U'}=\overline A(L_U),
 \qquad
 w_{U'}(\overline A P)=w_U(P)
\]
for every $P\in\PG(k-1,q^n)$.
\end{lemma}

Two nonzero $\F_q$-subspaces $A,B\leq \F_{q^n}$ are called scalar-equivalent if $A=\alpha B$ for some $\alpha\in \F_{q^n}^*$.  
For a prime power $q$, we use the notation
\[
 \qnum{m}{q}=\frac{q^m-1}{q-1}.
\]

The points of $\PG(1,q^n)$ are $
 P_\alpha=\langle(1,\alpha)\rangle_{\F_{q^n}}$ $(\alpha\in \F_{q^n})$,
$P_\infty=\langle(0,1)\rangle_{\F_{q^n}}$.
Let $j$ be a divisor of $n$.
A $(d-1)$-dimensional $\F_{q^j}$-\emph{subgeometry} of $\PG(k-1,q^n)$ is any set of points projectively equivalent to
\[
\set{\angles v_{\fqn}:v\in V,\,v\neq0},
\quad\text{where } V=\F_{q^j}^d\times\set0^{k-d}.
\]
Note that what here is subgeometry is called ``canonical subgeometry'' by some authors (e.g.\ \cite{Lun84}).
The canonical $\F_{q^j}$-subline of $\PG(1,q^n)$ is
\[
 \PG(1,q^j)=\set{\langle(1,\lambda)\rangle_{\F_{q^n}}:\lambda\in\F_{q^j}}\cup\set{\langle(0,1)\rangle_{\F_{q^n}}},
\]
where $j$ is a divisor of $n$.
For reference, we state the two questions from \cite{SZZ26} in the form needed below.

\begin{problem}\label{prob:szz1}
Determine whether there exists a $(q^j+1,i)$-regular fat linear set of rank $2i$ in $\PG(1,q^n)$ with $j>1$.
\end{problem}

\begin{problem}\label{prob:higher}
Construct an $(r,i)$-regular fat linear set in $\PG(k-1,q^n)$ for which
$
 k<r<\qnum{k}{q}$,
or for which the set of fat points is not an $\F_q$-subgeometry $\PG(k-1,q)$.
\end{problem}

\section{Scalar stabilizers and point weights}\label{sec:SS}

In this section, we describe the point weights by means of the scalar stabilizer of the defining subspace $U$.

\begin{lemma}\label{l:scalar-divisibility}
Let $0\neq U\leq_{\F_q}\fqn^k$.  There is a divisor $j$ of $n$ such that
$
 S(U)=\F_{q^j}$.
Moreover, $j$ divides $\dim_{\F_q}U$, and for every $\fqn$-linear subspace $W$ of ${\fqn^k}$, the intersection $U\cap W$ is an $\F_{q^j}$-vector space.  In particular,
\begin{equation}\label{eq:weight-div}
 j\mid w_U(P),
 \qquad
 \text{for all } P\in\PG(k-1,q^n).
\end{equation}
\end{lemma}

\begin{proof}
Since $U$ is an $\F_q$-space, $\F_q\subseteq S(U)$.  The set $S(U)$ is closed under addition, additive inverses, and multiplication.  If $0\neq a\in S(U)$, then multiplication by $a$ is an injective $\F_q$-linear endomorphism of the finite-dimensional space $U$.  It is therefore surjective, so $aU=U$ and $a^{-1}U=U$. 
Thus, $S(U)$ is a finite subfield of $\fqn$ containing $\F_q$.  By the subfield theorem for finite fields, it is $\F_{q^j}$ for a unique $j\mid n$; see \cite{LN97}.

The scalar action makes $U$ an $\F_{q^j}$-space, and hence
\[
 \dim_{\F_q}U=j\dim_{\F_{q^j}}U.
\]
 Thus, $U\cap W$ is an $\F_{q^j}$-space, and its dimension over $\fq$ is a multiple of $j$.
 Taking $W=P=\angles a_{\fqn}$ for $a\in\fqn^*$ proves \eqref{eq:weight-div}.
\end{proof}

\begin{corollary}\label{cor:weight-one-destroys}
If $S(U)=\F_{q^j}$ and $L_U$ contains a point of weight $i$, then $j\mid\gcd(i,n)$.
In particular, if $L_U$ contains a point of weight one, then
$
 S(U)=\F_q$.
\end{corollary}

Note that the field $S(U)$ depends on the vector subspace $U$, rather than merely the related linear set.
As discussed in \cite{CMP24}, a linear set can arise from inequivalent defining subspaces. The assertions in the next theorem can be deduced from \cite[Proposition 6.1]{SZZ26}, and its proof.

\begin{theorem}\cite{SZZ26}\label{thm:three-heavy}
Let $i\geq 1$ and let $U\leq_{\F_q}\F_{q^n}^2$ have dimension $2i$.
Suppose that $L_U$ contains three distinct points of weight $i$.
Then $U$ is $\GL(2,q^n)$-equivalent to $T\times T$
for an $i$-dimensional $\F_q$-subspace $T\leq_{\F_q}\F_{q^n}$. If $S(T)=\mathbb F_{q^j}$, then the points of weight $i$ form a projective \(\mathbb F_{q^j}\)-subline of \(\mathrm{PG}(1,q^n)\), and hence there are exactly \(q^j+1\) such points.
\end{theorem}

Given that weight-one points exist by definition, \cref{cor:weight-one-destroys} and \cref{thm:three-heavy} imply the following:

\begin{theorem}\label{thm:line-rigidity}
Let $i\geq 2$, let $U\leq_{\F_q}\fqn^2$ have dimension $2i$, and suppose that $L_U$ is $(r,i)$-regular fat.  Then
$
 r\leq 2$ or $r=q+1$.
If $r>2$, the fat points form a projective $\F_q$-subline.
\end{theorem}

\begin{proof}
Suppose that $r>2$.  By \cref{thm:three-heavy}, up to an $\fqn$-linear isomorphism, one has $U=T\times T$ with $S(T)=\F_{q^j}$, and the fat point set has $q^j+1$ points.  Every point weight is divisible by $j$.  Regular fatness provides a point of weight one, so $j=1$. Therefore, $r=q+1$, and the fat point set is projectively equivalent to $\PG(1,q)$.  \end{proof}

This gives a negative answer to the first open problem posed in \cite{SZZ26} (Problem~\ref{prob:szz1}).

Let us describe the normal forms for the rank-$(2i)$ $(r,i)$-regular fat linear sets in $\PG(1,q^n)$ for $r=2$ or $r=q+1$.

\begin{proposition}\label{prop:two-heavy}\cite[Proposition 3.2 and Theorem 4.1]{NPSZ22}
Let $U\leq_{\F_q}\fqn^2$ have dimension $2i$ and suppose that exactly two points have weight $i$.  
Up to an element of $\GL(2,q^n)$,
$
 U=T_1\times T_2$
for $i$-dimensional $\F_q$-subspaces $T_1,T_2\leq \fqn$.  The resulting linear set is $(2,i)$-regular fat if and only if
\begin{equation}\label{eq:two-heavy-condition}
\dim(T_1\cap\alpha T_2)\leq 1
 \qquad
 (\alpha\in \fqn).
\end{equation}
\end{proposition}

The last type of regular fat linear set is associated with the usual Sidon space intersection condition.

\begin{proposition}\label{prop:qplusone}
Let $U\leq_{\F_q}\fqn^2$ have dimension $2i$. Then $L_U$ is
$(q+1,i)$-regular fat if and only if, up to an element of
$\GL(2,q^n)$,
$U=T\times T$
for an $i$-dimensional $\F_q$-subspace $T\leq\fqn$ satisfying
\begin{equation}\label{eq:sidon-condition}
S(T)=\F_q,
\qquad
\dim(T\cap\alpha T)\leq1
\quad(\alpha\notin\F_q).
\end{equation}
In this case, the fat points form an $\F_q$-subline.
\end{proposition}

\begin{proof}
For $\alpha\in\fqn^*$  and any $\fq$-subspace $T$ of $\fqn$, we have
\begin{equation}\label{eq:line-weight1}
w_{T\times T}\bigl(\angles{(1,\alpha)}_{\fqn}\bigr)
=
\dim\set{(x,\alpha x):x\in T\cap\alpha^{-1}T}=
\dim(T\cap\alpha T),
\end{equation}
while the points of slopes $0$ and $\infty$ have weight $\dim T$.

Suppose first that $L_U$ is $(q+1,i)$-regular fat. By
\cref{thm:three-heavy}, up to $\GL(2,q^n)$-equivalence we may assume that
$U=T\times T$ for some $i$-dimensional $\F_q$-subspace
$T\leq\fqn$. 
Since the scalar stabilizer is preserved by $\GL(2,q^n)$, $S(T\times T)=S(T)$, and $L_U$ contains at least one point of weight one, \cref{cor:weight-one-destroys} yields $S(T)=\F_q$.
Now, if $\alpha\notin\F_q$ and
$\dim(T\cap\alpha T)>1$, then
\eqref{eq:line-weight1} and regular fatness force
$\dim(T\cap\alpha T)=i$. Thus
$T=\alpha T$, contradicting $\alpha\notin S(T)=\F_q$.

Conversely, let $U=T\times T$ with $T$ satisfying
\eqref{eq:sidon-condition}. The standard $\F_q$-subline consists
entirely of points of weight $i$, whereas \eqref{eq:line-weight1}
shows that every point outside this subline has weight at most one, and $L_U$ is
$(q+1,i)$-regular fat.
\end{proof}

Condition \eqref{eq:sidon-condition} is a standard intersection form of the Sidon-space property; see \cite{BSZ17,RRT18,Zul23}.  
Thus the rank-$2i$ $(q+1,i)$-regular fat linear sets are related to the existence and equivalence theory of Sidon spaces, while the  rank-$2i$ $(2,i)$-regular fat linear sets are controlled by a two-space multi-Sidon condition.

\section{Admissible block products}\label{sec:ABP}
 
Let $T=(T_1,\ldots,T_c)$, where
$
 T_1,\ldots,T_c
$
are  $i$-dimensional $\fq$-subspaces of $\fqn$, and write
$ S(T_\nu)=\F_{q^{j_\nu}}$.
Also, let $m=(m_1,\ldots,m_c)$, where  $m_1,\ldots,m_c$ are positive integers such that $
 k=m_1+\cdots+m_c\geq 2$.
Partition the coordinate set of $\PG(k-1,q^n)$ into blocks $I_1,\ldots,I_c$ with $\card{I_\nu}=m_\nu$. The \emph{block product} related to $(T,m)$ is 
\begin{equation}\label{eq:block-product}
 U(T,m)=T_1^{m_1}\times\cdots\times T_c^{m_c}\leq \fqn^k.
\end{equation}

The following conditions are exactly what is needed to prevent intermediate point weights.

\begin{definition}\label{def:admissible}
The block product $U(T,m)$ in \eqref{eq:block-product} is \emph{admissible} if:
\begin{enumerate}[label=\textup{(A\arabic*)}]
\item for every $\nu$ with $m_\nu\geq 2$ and every $\alpha\in\fqn\setminus S(T_\nu)$,
\[
 \dim_{\F_q}(T_\nu\cap\alpha T_\nu)\leq 1;
\]
\item for $\mu\neq\nu$ and every $\alpha\in \fqn$,
\[
 \dim_{\F_q}(T_\mu\cap\alpha T_\nu)\leq 1;
\]
\item if $c=1$, then $S(T_1)=\F_q$.
\end{enumerate}
\end{definition}

Each scalar-equivalence class produces one finite subgeometry in its coordinate span.

\begin{theorem}\label{thm:block-construction}
Let $U(T,m)$ be an admissible block product.  
The linear set associated with $U=U(T,m)$ in \eqref{eq:block-product} is $(r,i)$-regular fat, where
\begin{equation}\label{eq:block-r}
 r=\sum_{\nu=1}^c\qnum{m_\nu}{q^{j_\nu}}.
\end{equation}
Its fat point set is a disjoint union
\begin{equation}\label{eq:block-heavy}
 \cF(U)=\bigsqcup_{\nu=1}^c F_\nu,
\end{equation}
where each $F_\nu\cong\PG(m_\nu-1,q^{j_\nu})$ is a subgeometry of $\PG(k-1,q^n)$ lying in the $(m_\nu-1)$-dimensional coordinate subspace   supported on $I_\nu$.
\end{theorem}

\begin{proof}
Let $P=\angles{a}_{\fqn}$, $a=(a_1,\ldots,a_k)$, be a point in $\PG(k-1,q^n)$.  Put
\begin{equation}\label{eq:support}
 J=\supp(a)=\set{h:h\in\{1,\ldots,k\},\,a_h\neq 0},
\end{equation}
and let $\beta(h)$ be the block containing coordinate $h$.  
Define the following subset of $\fqn$:
\begin{equation}\label{eq:block-fiber}
 \Lambda_P=\bigcap_{h\in J}a_h^{-1}T_{\beta(h)}.
\end{equation}
The map $\lambda\mapsto\lambda a$ is an $\F_q$-linear isomorphism from $\Lambda_P$ onto $U\cap \angles{a}_{\fqn}$, so
\begin{equation}\label{eq:block-weight}
 w_U(P)=\dim\Lambda_P.
\end{equation}
Every space in the intersection \eqref{eq:block-fiber} has dimension $i$.  Therefore $w_U(P)=i$ if and only if all participating subspaces are equal.

If the support of $P$ meets two distinct blocks, equality of two participating spaces would make the corresponding $T_\mu,T_\nu$ scalar-equivalent, contrary to the choice of representatives.  Hence a fat point is supported in a single block $I_\nu$.  Within that block, fix $h_0\in J$.  The spaces $a_h^{-1}T_\nu$ are all equal if and only if
\[
 \frac{a_h}{a_{h_0}}T_\nu=T_\nu
 \qquad
 (h\in J),
\]
which is equivalent to all nonzero coordinate ratios lying in $S(T_\nu)^*=\F_{q^{j_\nu}}^*$.  Thus the fat points in that block form the embedded subgeometry $\PG(m_\nu-1,q^{j_\nu})$.  This proves \eqref{eq:block-heavy} and \eqref{eq:block-r}.

Now let $P$ be nonfat.  If its support meets distinct blocks, then $\Lambda_P$ is contained in an intersection covered by condition \textup{(A2)}, so its dimension is at most one.  If its support lies in one block, then at least two participating scalar copies differ, and the corresponding scalar ratio lies outside $S(T_\nu)$.  Condition \textup{(A1)} again gives $\dim\Lambda_P\leq 1$.  Hence every nonfat point of $L_U$ has weight one.

It remains to exhibit such a point.  If $c\geq 2$, choose nonzero $x\in T_\mu$ and $y\in T_\nu$ in distinct blocks, and form a vector supported in one coordinate of each block.  Its point belongs to $L_U$ and is nonfat, hence has weight one.  If $c=1$, choose $\F_q$-linearly independent $x,y\in T_1$ and place them in two coordinates.  Since $S(T_1)=\F_q$, the resulting point is nonfat and therefore has weight one.
\end{proof}

\begin{remark}
Suppose that $T_1,\ldots,T_c$ satisfy the admissibility conditions and $S(T_\nu)=\F_q$ for every $\nu$.  
By Theorem~\ref{thm:block-construction}, in the extremal cases $c=k$, and $c=1$,
the fat point set is a projective  basis, and an $\fq$-subgeometry of $\PG(k-1,q^n)$, isomorphic to $\PG(k-1,q)$, respectively.
\end{remark}

\section{Structural constraints on admissible block products}\label{sec:struc}

A closer inspection of the proof of \cref{thm:block-construction} reveals that the admissible block products introduced in \cref{def:admissible} are subject to additional structural restrictions. In this section, we provide a detailed description of these conditions. 
We derive two properties from the proof of \cref{thm:block-construction} that will be useful in this section.
\begin{proposition}\label{p:derived}
Let $U = U(T, m)$ be an admissible block product, with $L_U$ the associated linear set in $\PG(k-1, q^n)$. 
Furthermore, let $P =\angles a_{\fqn}$, $a=(a_1,\ldots,a_k)$, be a point in $L_U$ and let $J$ be its support as defined in \eqref{eq:support}. The following holds:
\begin{enumerate}[label={$(\roman*)$}]
\item    If $J$ intersects two distinct blocks, then $P$ has weight one.
\item     If $J$ is contained in a single block $I_\nu$, then $P$ has weight $i$ if and only if the quotients of all non-zero coordinates are in the scalar stabilizer $\F_{q^{j_\nu}}$.
 \item   In any case, the following formula holds for the weight of a point $P$: 
 \begin{equation}\label{eq:derived}
 w_U(P) = \dim \bigcap_{h\in J}a_h^{-1}T_{\beta(h)},
 \end{equation}
where $\beta(h)$ is the block containing coordinate $h$.
 \end{enumerate}
 \end{proposition}

The dimensions of the scalar stabilizers of distinct components cannot share a nontrivial divisor.

\begin{proposition}\label{prop:pairwise-coprime}
In the setting of \cref{sec:ABP},
$ \gcd(j_\mu,j_\nu)=1$ ($\mu\neq\nu$).
\end{proposition}

\begin{proof}
Fix distinct blocks $I_\mu$ and $I_\nu$.  Choose nonzero elements
$
 x\in T_\mu$,
$ y\in T_\nu$,
and coordinates $r\in I_\mu$, $s\in I_\nu$.  The vector
$
 v=xe_r+ye_s
$
belongs to $U=U(T,m)$, so the projective point $P=\angles v_{\fqn}$ lies in $L_U$.  Its support meets two different scalar-equivalence blocks; hence, $P$ is not fat.  Regular fatness therefore gives
$
 w_U(P)=1.
$
Therefore, by \cref{p:derived}, the dimension of
$
 \Lambda_P=x^{-1}T_\mu\cap y^{-1}T_\nu$
is one.
We have
\[
 S(T_\mu)\cap S(T_\nu)
 =\F_{q^{j_\mu}}\cap\F_{q^{j_\nu}}
 =\F_{q^g},
\]
where $g=\gcd(j_\mu,j_\nu)$.
Multiplying a subspace of $\fqn$ by a nonzero scalar does not change its scalar stabilizer, so both $x^{-1}T_\mu$ and $y^{-1}T_\nu$ are stable under $\F_{q^g}$.  Their intersection $\Lambda_P$ is therefore an $\F_{q^g}$-vector space.  Consequently
\[
 g\mid\dim\Lambda_P=w_U(P)=1.
\]
Thus $g=1$, as claimed.
\end{proof}

A nontrivial scalar stabilizer can occur in a repeated block only at the full fatness degree.

\begin{proposition}\label{prop:repeated-nontrivial}
In the setting of \cref{sec:ABP}, suppose that $m_\nu\geq 2$ and $j_\nu>1$. Then $j_\nu=i$,
and $T_\nu$ is one-dimensional over $\F_{q^i}$.  In particular, $T_\nu=\beta\F_{q^i}$ for some $\beta\in \fqn^*$.
\end{proposition}

\begin{proof}
Since $S(T_\nu)=\F_{q^{j_\nu}}$, the space $T_\nu$ is a vector space over $\F_{q^{j_\nu}}$.  Write
\[
 d=\dim_{\F_{q^{j_\nu}}}T_\nu,
\]
so $d=i/j_\nu$.
Assume, for a contradiction, that $d\geq2$.  Then we may choose $x,y\in T_\nu$ that are linearly independent over $\F_{q^{j_\nu}}$.  Put $\alpha=y/x$.
The linear independence yields $\alpha\notin S(T_\nu)=\F_{q^{j_\nu}}$.
On the other hand, $x\in T_\nu$ and $\alpha x=y\in T_\nu$, so
$ 0\neq y\in T_\nu\cap\alpha T_\nu$.
Also $T_\nu\cap\alpha T_\nu$ is a vector space over $S(T_\nu)=\F_{q^{j_\nu}}$.
This implies that  its dimension over $\fq$ is a positive multiple of $j_\nu$, contradicting (A1) in the definition of an admissible block product.

Therefore $i=j_\nu$.
A one-dimensional vector space over $\F_{q^i}$ has the form $\beta\F_{q^i}$ for some nonzero $\beta\in T_\nu$, which proves the final assertion.
\end{proof}

Several useful restrictions follow immediately from the two preceding propositions.

\begin{corollary}\label{cor:arithmetic-consequences}
For an admissible block product, the following statements hold.
\begin{enumerate}[label=\textup{(\roman*)}]
\item If $\gcd(i,n)=1$, then every scalar stabilizer degree $j_\nu$ equals one.
\item If a repeated block has scalar stabilizer degree greater than one, then its degree is $i$, and every other block has scalar stabilizer degree one.
\item If $i$ is prime, at most one block has nontrivial scalar stabilizer.
\end{enumerate}
\end{corollary}

\begin{proof}
Every $j_\nu$ divides both $i$ and $n$ by \cref{l:scalar-divisibility}, which proves the first assertion.  For the second, \cref{prop:repeated-nontrivial} gives $j_\nu=i$.  If $\mu\neq\nu$, then $j_\mu\mid i$ and \cref{prop:pairwise-coprime} gives $\gcd(i,j_\mu)=1$, so $j_\mu=1$.  If $i$ is prime, every $j_\nu$ is either one or $i$, and pairwise coprimality permits at most one occurrence of $i$.
\end{proof}

A single scalar-equivalence block is forced back to the ground field.

\begin{corollary}\label{cor:one-block}
If an admissible block product has one scalar-equivalence block, then its defining subspace is projectively equivalent to $T^k$ with $S(T)=\F_q$, and its fat point set is the subgeometry $\PG(k-1,q)$.
\end{corollary}

\begin{proof}
Condition \textup{(A3)} in \cref{def:admissible} gives $S(T)=\F_q$.  By \cref{p:derived}~$(ii)$, the fat point set is $\PG(k-1,q)$.
\end{proof}

The block formula \eqref{eq:block-r} also gives simple bounds for the number of fat points.

\begin{proposition}\label{prop:saturated-lower-bound}
For any rank-$ki$ $(r,i)$-regular fat linear set in $\PG(k-1,q^n)$ defined by an admissible block product, we have
$ r\geq k$.
Equality holds if and only if every block has multiplicity one.  
Furthermore,
\begin{equation}\label{eq:bbound}
 r\le\qnum{k}{q^{d^*}}, \text{ where }d^*=\gcd(i,n).
\end{equation}
Equality holds if and only if $\gcd(i,n)=1$ and the admissible block product has one scalar-equivalence block.
\end{proposition}

\begin{proof}
For every prime power $Q$ and every $m\geq 1$,
$
 \qnum{m}{Q}=1+Q+\cdots+Q^{m-1}\geq m
$, with equality exactly when $m=1$.  Applying this to each block gives
\[
 r=\sum_{\nu=1}^c\qnum{m_\nu}{q^{j_\nu}}
 \geq\sum_{\nu=1}^c m_\nu
 =k.
\]
  For positive integers $a,b$,
$
 \qnum{a+b}{q}>\qnum{a}{q}+\qnum{b}{q}.
$
Since every $j_\nu$ divides both $i$ and $n$ by \cref{l:scalar-divisibility}, \cref{eq:block-r} implies
\[
 r=\sum_{\nu=1}^c\qnum{m_\nu}{q^{j_\nu}}\le
 \sum_{\nu=1}^c\qnum{m_\nu}{q^{d^*}}
 \le\qnum{m_1+\cdots+m_c}{q^{d^*}}=\qnum k{q^{d*}}.
\]
If equality holds, then $c=1$.
In this case, \cref{cor:one-block} implies $r=\qnum kq$.
Conversely, if the admissible block product has one scalar-equivalence block and $\gcd(i,n)=1$, then \cref{cor:one-block} implies $r=\qnum kq=\qnum k{q^{d^*}}$.
\end{proof}

\section{Explicit families and a higher-dimensional construction}\label{sec:fam}

Let $\theta\in \F_{q^4}^*$ be primitive.  Set
\begin{equation}\label{eq:theta-pair}
 T=\angles{1,\theta}_{\F_q},
 \qquad
 R=\angles{1,\theta^2}_{\F_q}.
\end{equation}

The two factors have no scalar field larger than $\F_q$:

\begin{lemma}\label{lem:theta-stabilizers}
For the spaces in \eqref{eq:theta-pair},
$ S(T)=S(R)=\F_q$.
\end{lemma}

\begin{proof}
Since $\theta$ is primitive, it does not lie in $\F_q$, and hence $T$ has dimension two.  We first note that $\theta^2\notin\F_{q^2}$.  Otherwise
$ \theta^{2(q^2-1)}=1$,
so $q^4-1$ would divide $2(q^2-1)$.  Since $q^4-1=(q^2-1)(q^2+1)$, this would force $q^2+1\mid 2$, which is impossible for $q\geq 2$.  In particular, $R$ also has dimension two.

For either $A=T$ or $A=R$, the degree of $S(A)$ over $\F_q$ divides both $2$ and $4$.  Thus $S(A)$ is $\F_q$ or $\F_{q^2}$.  In the latter case $A$ is one-dimensional over $\F_{q^2}$, so $A=\beta\F_{q^2}$ for some $\beta\in \F_{q^4}^*$.  Since $1\in A$, one has $\beta\in\F_{q^2}^*$ and therefore $A=\F_{q^2}$.  This contradicts $\theta\notin\F_{q^2}$ when $A=T$ and $\theta^2\notin\F_{q^2}$ when $A=R$.
\end{proof}

The factors also lie in distinct scalar orbits.

\begin{lemma}\label{lem:theta-nonequivalent}
There is no $\alpha\in \F_{q^4}$ such that
$T=\alpha R$.
\end{lemma}

\begin{proof}
Suppose that $T=\alpha R$.
Clearly, $\alpha\neq0$. Since $1,\theta\in T$, there are $a,b,c,d\in\F_q$ such that
$ 1=\alpha(a+b\theta^2)$, $\theta=\alpha(c+d\theta^2)$.
Multiplying the first equality by $\theta$ and comparing with the second gives
\[
 b\theta^3-d\theta^2+a\theta-c=0.
\]
The polynomial $bX^3-dX^2+aX-c$ is nonzero and has degree at most three.  This contradicts $[\F_q(\theta):\F_q]=4$.
\end{proof}

Because every factor has dimension two, the preceding lemmas immediately yield all required intersection bounds.

\begin{corollary}\label{cor:theta-intersections}
For every $\alpha\in \F_{q^4}$,
\[
 \dim(T\cap\alpha T)\leq 1
 \quad\text{if }\alpha\notin\F_q,
\]
\[
 \dim(R\cap\alpha R)\leq 1
 \quad\text{if }\alpha\notin\F_q,
\]
and
\[
 \dim(T\cap\alpha R)\leq 1.
\]
\end{corollary}

\begin{proof}
An intersection of two-dimensional spaces has dimension greater than one only when the spaces coincide.  The first two assertions follow from \cref{lem:theta-stabilizers}, and the third follows from \cref{lem:theta-nonequivalent}.
\end{proof}

\begin{corollary}\label{cor:op2-all-k-i2}
For every prime power $q$ and every $k\geq2$, put
\[
 U_k=T^{k-1}\times R\leq \F_{q^4}^k,
\]
where $T$ and $R$ are given by \eqref{eq:theta-pair}.  Then $L_{U_k}$ is a rank-$2k$
\[
 \left(\qnum{k-1}{q}+1,2\right)
\]
regular fat linear set in $\PG(k-1,q^4)$.  Its fat point set is an $\F_q$-subgeometry $\PG(k-2,q)$ together with one point outside its span, and
$k<\qnum{k-1}{q}+1<\qnum{k}{q}$.
In particular, both alternatives in Problem 2.5 occur for every $k\geq3$.
\end{corollary}

\begin{proof}
By Corollary~\ref{cor:theta-intersections}, for any $k\ge2$, $T^{k-1}\times R$ is an admissible block product.
Therefore, Theorem~\ref{thm:block-construction} yields all assertions in Corollary~\ref{cor:op2-all-k-i2}, up to the last one.
Finally,
\[
 \left(\qnum{k-1}{q}+1\right)-k=\qnum{k-1}{q}-(k-1)=\sum_{t=0}^{k-2}(q^t-1)>0
\]
for $k\geq3$, while
\[
 \qnum{k}{q}-\left(\qnum{k-1}{q}+1\right)=q^{k-1}-1>0.
\]
This proves the two strict inequalities.
\end{proof}

More generally, the same argument with arbitrary positive multiplicities $a,b$ gives
\[
 U=T^a\times R^b\leq \F_{q^4}^{a+b},
\]
whose fat point set is
$\PG(a-1,q)\sqcup\PG(b-1,q)$
in complementary coordinate subspaces and whose number of fat points is
$r=\qnum{a}{q}+\qnum{b}{q}$.
This is the two-block strengthening of the family above and follows directly from \cref{thm:block-construction} together with \cref{lem:theta-stabilizers,lem:theta-nonequivalent}.

\begin{remark}
Let us isolate the smallest instance in dimension higher than two.
Let  $T$ and $R$ be the spaces in \eqref{eq:theta-pair}, and put
\[
 U_3=T\times T\times R\leq \F_{q^4}^3.
\]
Then $\dim_{\F_q}U_3=6$, and $L_{U_3}\subseteq\PG(2,q^4)$ is a rank-six $(q+2,2)$-regular fat linear set.  Its fat point set is
an $\F_q$-subline together with one point outside that line.

Already this basic case answers Problem \ref{prob:higher} in projective dimension two: here $k=3$ and $r=q+2$, so
\[
 3<q+2<q^2+q+1=\qnum{3}{q},
\]
and the set of fat points is not an $\F_q$-subgeometry $\PG(2,q)$.
\end{remark}

Now we are going to construct rank-$ki$ $(r,i)$-regular fat linear sets in $\PG(k-1,q^{2i})$, using a family constructed  in \cite[Corollary~3.11]{Zul23}.  

\begin{proposition}\label{prop:zullo-input}
Assume $q\geq 3$ and $i\geq 2$.  Put
\[
 \tau=\left\lfloor\frac{q-1}{2}\right\rfloor.
\]
Choose an integer $s$ with $\gcd(s,i)=1$, a primitive element $w\in \F_{q^{i}}^*$, and an element $z\in \F_{q^{i}}$ such that $X^2+zX+w$ is irreducible over $\F_{q^{i}}$.  Let $\gamma_0\in \F_{q^{2i}}$ be a root and set
\[
 \gamma_t=w^t\gamma_0,
 \qquad
 V_t=\set{u+u^{q^s}\gamma_t:u\in \F_{q^{i}}}
 \quad
 (0\leq t<\tau).
\]
Each $V_t$ has $\F_q$-dimension $i$, and
$
 S(V_t)=\F_q$.
Moreover,
\[
 \dim(V_t\cap\alpha V_t)\leq 1
 \quad
 (\alpha\in\F_{q^{2i}}\setminus\fq),
\]
\[
 \dim(V_t\cap\alpha V_u)\leq 1
 \quad
 (t\neq u,\ \alpha\in \F_{q^{2i}}),
\]
and
\[
 \dim(\F_{q^i}\cap\alpha V_t)\leq 1
 \quad
 (\alpha\in \F_{q^{2i}}).
\]
\end{proposition}

\begin{proof}
This proposition is essentially a rephrasing of \cite[Corollary 3.11 and Proposition 3.12]{Zul23} in terms of multi-orbit Sidon sets. Therefore, the reader is referred to \cite{Zul23} to understand the rest of this proof.

The map $u\mapsto u+u^{q^s}\gamma_t$ is injective, because a nonzero element in its kernel would place $\gamma_t$ in $\F_{q^{i}}$.  Hence $\dim_{\F_q}V_t=i$.  
In  \cite[Corollary~3.11]{Zul23} it is shown that the scalar orbits of the $V_t$ are distinct full-length orbits and that their union has minimum subspace distance $2i-2$.  Since the subspace distance equals
\[
 d(A,B)=2i-2\dim(A\cap B)
\]
for $i$-dimensional subspaces, the minimum-distance assertion is equivalent to the first two intersection bounds.  Full orbit length means that the scalar stabilizer is $\F_q^*$, and hence $S(V_t)=\F_q$.  
Finally, the intersection estimate between $\F_{q^{i}}$ and every scalar translate of $V_t$ has dimension at most one by \cite[Proposition~3.12]{Zul23}.
\end{proof}

The multiplicities of these scalar orbits may be chosen independently.  The subfield factor can also be repeated, producing a component over $\F_{q^i}$.

\begin{theorem}\label{thm:zullo-multiplicity}
Retain the notation of \cref{prop:zullo-input}.  Let
\[
 m_0,\ldots,m_{\tau-1},b\in\mathbb Z_{\geq 0},
 \qquad
 k=b+\sum_{t=0}^{\tau-1}m_t.
\]
Assume that $k\geq 2$ and that at least one $m_t$ is positive.  Form the product of the nonzero blocks
\[
 U=\prod_{t=0}^{\tau-1}V_t^{m_t}\times \F_{q^{i}}^b\leq \F_{q^{2i}}^k.
\]
Then $L_U$ is a rank-$ki$ $(r,i)$-regular fat linear set in $\PG(k-1,q^{2i})$.
The set of fat points is
\[
 \cF(U)=
 \left(\bigsqcup_{m_t>0}\PG(m_t-1,q)\right)
 \sqcup
 \begin{cases}
 \PG(b-1,q^i),&b>0,\\
 \varnothing,&b=0,
 \end{cases}
\]
and
\begin{equation}\label{eq:zullo-multiplicity-r}
 r=\frac{q^{bi}-1}{q^i-1}+\sum_{t=0}^{\tau-1}\frac{q^{m^t}-1}{q-1}.
\end{equation}
\end{theorem}

\begin{proof}
The self-intersection and cross-intersection conditions for the $V_t$ are given by \cref{prop:zullo-input}.  
Any two distinct scalar translates of $\F_{q^{i}}$ meet trivially.  The cross-intersections between $\F_{q^{i}}$ and the $V_t$ are again bounded by one by \cref{prop:zullo-input}. 
Thus the nonzero blocks satisfy conditions \textup{(A1)} and \textup{(A2)} of Definition~\ref{def:admissible}.

If at least two block types occur, condition \textup{(A3)} is automatic.  If only one block type occurs, it must be some $V_t$ because at least one $m_t$ is positive; then its multiplicity is $k\geq 2$ and its scalar stabilizer is $\F_q$.  Hence $U$ is an admissible block product.
The result follows from \cref{thm:block-construction}.
\end{proof}

A single repeated $V_t$ block and one subfield point give the same geometry as in \cref{cor:op2-all-k-i2}.

\begin{corollary}\label{cor:op2-all-k-general-i}
Assume $q\geq 3$, $i\geq 2$, and $k\geq 3$.  There exists a rank-$ki$
\[
 \left(\qnum{k-1}{q}+1,i\right)
\]
regular fat linear set in $\PG(k-1,q^{2i})$ whose fat points consist of an $\F_q$-subgeometry $\PG(k-2,q)$ together with one external point.
\end{corollary}

\begin{proof}
In \cref{thm:zullo-multiplicity}, take $m_0=k-1$, $b=1$, and all remaining multiplicities zero.
\end{proof}

\section{Uniqueness of the construction by admissible block products}\label{sec:unique}

The block construction has a converse once the rank reaches the natural value $ki$ and the fat points span the ambient projective space.
A vector basis made by fat points forces the defining subspace to split coordinatewise.

\begin{theorem}\label{thm:saturated-converse}
Let $U\leq\fqn^k$, $k\ge2$, have dimension $ki$, and suppose that $L_U$ is an $(r,i)$-regular fat linear set whose fat points span $\PG(k-1,q^n)$.
After a linear isomorphism, there are $i$-dimensional subspaces
$T_1$, $\ldots$, $T_c\leq \fqn$
and positive integers $m_1,\ldots,m_c$ with $m_1+\cdots+m_c=k$ such that
\begin{equation}\label{eq:saturated-product}
 U=T_1^{m_1}\times\cdots\times T_c^{m_c},
\end{equation}
which is an admissible block product.  Consequently, if $S(T_\nu)=\F_{q^{j_\nu}}$, then
\[
 \cF(U)=\bigsqcup_{\nu=1}^c\PG(m_\nu-1,q^{j_\nu})
\]
and
\[
 r=\sum_{\nu=1}^c\qnum{m_\nu}{q^{j_\nu}}.
\]
\end{theorem}

\begin{proof}
Because $\cF(U)$ spans $\PG(k-1,q^n)$, choose $k$ projectively independent fat points.  By a projectivity we may send them to the coordinate points
$\angles{e_1}_{\fqn}$, $\ldots$, $\angles{e_k}_{\fqn}$.
For each coordinate define
\[
 T_a=\set{x\in \fqn:xe_a\in U}
 \qquad
 (1\leq a\leq k).
\]
The point $\angles{e_a}_{\fqn}$ is fat, so $U\cap\angles{e_a}_{\fqn}=T_ae_a$ has $\F_q$-dimension $i$.  Hence
$\dim T_a=i$ ($1\leq a\leq k$).
The coordinate subspaces $T_ae_a$ have pairwise trivial intersection, and therefore their sum is direct.  Therefore,
$
 T_1e_1\oplus\cdots\oplus T_ke_k= U$,
and
\begin{equation}\label{eq:coordinate-product}
 U=T_1\times\cdots\times T_k.
\end{equation}

We now group the factors $T_1,\ldots,T_k$ by scalar equivalence.  For each equivalence class choose one representative, denoted again by $T_\nu$, and let $m_\nu$ be the number of coordinates in that class.  After permuting coordinates, the members of each class occur consecutively.  If a coordinate factor in the $\nu$th class is $\beta T_\nu$, multiplication of that coordinate by $\beta^{-1}$ is induced by a diagonal element of $\GL(k,q^n)$.  Applying these diagonal scalings simultaneously replaces every factor in the class by the chosen representative.  Hence \eqref{eq:coordinate-product} becomes \eqref{eq:saturated-product}, with the representatives $T_1,\ldots,T_c$ pairwise scalar-inequivalent by construction.

Keep the notation of \cref{thm:block-construction}: for $P=\angles a_{\fqn}$, $a=(a_1,\cdots,a_k)$, and $J=\supp(a)$, \cref{eq:block-fiber}, i.e.\ $\Lambda_P=\bigcap_{h\in J}a_h^{-1}T_{\beta(h)}$, gives \eqref{eq:block-weight}:
$w_U(P)=\dim\Lambda_P$.

It remains to prove that the block data are admissible.  We verify the three conditions separately.

For condition \textup{(A1)}, fix a block $I_\nu$ with $m_\nu\geq2$.  Suppose, for a contradiction, that there exists
$\alpha\notin S(T_\nu)$ with $\dim(T_\nu\cap\alpha T_\nu)>1$.
Choose two distinct coordinate positions $r,s\in I_\nu$ and consider the projective point $P_a=\angles{a}_{\fqn}$ represented by
$a=e_r+\alpha^{-1}e_s$.
Its support is exactly $\{r,s\}$.
By \eqref{eq:block-fiber} and \eqref{eq:block-weight},
\[
\Lambda_{P_a}
 =T_\nu\cap(\alpha^{-1})^{-1}T_\nu
 =T_\nu\cap\alpha T_\nu,\quad\text{and}\quad
  w_U(P_a)=\dim(T_\nu\cap\alpha T_\nu)>1.
\]
On the other hand, $P_a$ is not fat: if it had weight $i$, then the two participating copies would coincide, so $T_\nu=\alpha T_\nu$, which would imply $\alpha\in S(T_\nu)^*$, contrary to the choice of $\alpha$. Hence, $P_a$ is a nonfat point of $L_U$ with weight greater than one, contradicting regular fatness.  Therefore condition \textup{(A1)} holds.

For condition \textup{(A2)}, let $\mu\neq\nu$ and suppose that for some $\alpha\in \fqn^*$,
$\dim(T_\mu\cap\alpha T_\nu)>1$.
Choose a coordinate $r\in I_\mu$ and a coordinate $s\in I_\nu$, and again take
$a=e_r+\alpha^{-1}e_s$.
Now
$\Lambda_{P_a}=T_\mu\cap\alpha T_\nu$,
so $P_a\in L_U$ has weight greater than one.  The point cannot be fat, because weight $i$ would force $T_\mu=\alpha T_\nu$, contradicting the scalar-inequivalence of the two block representatives.  This is again impossible in a regular fat linear set.  Hence condition \textup{(A2)} holds.

Finally consider condition \textup{(A3)}.  If $c\geq2$, it holds by definition.  Suppose $c=1$.  Then \eqref{eq:saturated-product} has the form
$U=T_1^k$.
The scalar stabilizer of this product is exactly the scalar stabilizer of one factor:
$
 S(U)=S(T_1)
$,
because a scalar preserves $T_1^k$ if and only if it preserves each coordinate copy of $T_1$.  Since regular fatness requires a point of weight one, \cref{l:scalar-divisibility} forces $S(U)=\F_q$, and therefore $S(T_1)=\F_q$.  Moreover $k\geq2$, by assumption.

All three admissibility conditions have now been verified, completing the proof.
\end{proof}

\section{Semilinear equivalence of admissible block products}\label{sec:semil}

We first recover the block components and their spans from the
projective incidence structure of the fat point set.  We then classify
semilinear equivalence of the defining subspaces and determine their
stabilizers.

Let $\cF$ be the fat point set of an admissible block product.  Define a
graph $\Gamma(\cF)$ with vertex set $\cF$, in which two distinct vertices
are adjacent if their projective line contains a third point of $\cF$.
For each coordinate block $I_\nu$, write
\[
 V_\nu=\Span_{\fqn}\{e_h:h\in I_\nu\}\leq \fqn^k,
\]
and let $\PG(V_\nu)$ denote its projectivization.

\begin{lemma}\label{lem:incidence-components}
Let
\[
 \mathcal F=\bigsqcup_{\nu=1}^c F_\nu,
 \qquad
 F_\nu\cong\PG(m_\nu-1,q^{j_\nu}),
\]
be the fat point set in \eqref{eq:block-heavy}.  The connected components
of $\Gamma(\mathcal F)$ are precisely $F_1,\ldots,F_c$.  Moreover, $F_\nu$
spans $\PG(V_\nu)$ for every $\nu$.  
\end{lemma}

\begin{proof}
Suppose that $m_\nu\geq2$.  Any two distinct points of $F_\nu$ lie on
an $\F_{q^{j_\nu}}$-subline contained in $F_\nu$.  This subline has
$q^{j_\nu}+1\geq3$ points, so the two vertices are adjacent.  Hence
$F_\nu$ induces a complete connected component.  When $m_\nu=1$, the
set $F_\nu$ consists of a single vertex.

Now take $P\in F_\mu$ and $Q\in F_\nu$, where $\mu\neq\nu$.
Choose nonzero representatives $u\in V_\mu$ and $v\in V_\nu$.
Every point of the projective line $PQ$ other than its endpoints has
a representative $au+bv$ with $a,b\in \fqn^*$. Since the coordinate block spans form a direct sum, this vector has nonzero coordinates
in both blocks.  Such a point is not heavy by
Theorem~\ref{thm:block-construction}.  Therefore no edge joins
distinct sets $F_\mu$ and $F_\nu$, proving the assertion about
connected components.

Finally, $F_\nu$ contains all coordinate points indexed by $I_\nu$.
These points span $\PG(V_\nu)$, so the projective spans are recovered
from the components as well.
\end{proof}

Consequently, any bijective semilinear map taking $\cF(U)$ onto
$\cF(U')$ induces a bijection between block components and block spans.

\begin{theorem}\label{thm:semilinear-equivalence}
Let
\[
 U=T_1^{m_1}\times\cdots\times T_c^{m_c},
 \qquad
 U'=(T'_1)^{m'_1}\times\cdots\times(T'_{c'})^{m'_{c'}}
\]
be admissible block products in $\fqn^k$, with all factors of
$\F_q$-dimension $i$. 
%and with pairwise scalar-inequivalent factors
%within each product.  
Fix $\sigma\in\Aut(\fqn)$.

There exists $A\in\GL(k,\fqn)$ such that the semilinear map
$\Phi(x)=A(x^\sigma)$ satisfies $\Phi(U)=U'$ if and only if $c=c'$
and there are a permutation $\pi$ of $\{1,\ldots,c\}$ and scalars
$a_1,\ldots,a_c\in \fqn^*$ such that
\begin{equation}\label{eq:equivalence-data}
 m'_{\pi(\nu)}=m_\nu,
 \qquad
 T'_{\pi(\nu)}=a_\nu T_\nu^\sigma
 \quad (1\leq\nu\leq c).
\end{equation}
After ordering the target blocks according to $\pi$, the matrix $A$
is block diagonal, with its $\nu$th block of the form
\begin{equation}\label{eq:block-matrix-form}
 a_\nu B_\nu,
 \qquad
 B_\nu\in\GL(m_\nu,S(T_\nu)).
\end{equation}
Conversely, every choice of data in \eqref{eq:equivalence-data} and
matrices in \eqref{eq:block-matrix-form} defines such an equivalence.
\end{theorem}

\begin{proof}
Suppose first that $\Phi(U)=U'$.  Semilinear maps preserve point
weights, so the induced projective map takes $\cF(U)$ onto $\cF(U')$.
By Lemma~\ref{lem:incidence-components}, it maps connected components
to connected components and hence induces a bijection between their projective spans.
Thus $c=c'$, and there is a permutation $\pi$ such that
\[
 \Phi(V_\nu)=V'_{\pi(\nu)}.
\]
Since $V_\nu$ is the span of some base points, $V_\nu^\sigma=V_\nu$.
Therefore $A(V_\nu)=V'_{\pi(\nu)}$, and preservation of
dimension gives $m_\nu=m'_{\pi(\nu)}$.

Fix $\nu$ and write $C_\nu=(c_{rs})$ for the corresponding invertible
block of $A$.  Since $A(U^\sigma)=U'$, a vector supported in the
$s$th coordinate of this block, with entry in $T_\nu^\sigma$, is
mapped to a vector whose entries belong to $T'_{\pi(\nu)}$.  Hence
\[
 c_{rs}T_\nu^\sigma\subseteq T'_{\pi(\nu)}
 \qquad\text{for all }r,s.
\]
Choose a nonzero coefficient $a_\nu$ of $C_\nu$.  Both factors have
$\F_q$-dimension $i$, so this inclusion gives
\[
 T'_{\pi(\nu)}=a_\nu T_\nu^\sigma.
\]
It follows that
\[
 a_\nu^{-1}c_{rs}T_\nu^\sigma\subseteq T_\nu^\sigma,
\]
and thus every entry of $B_\nu=a_\nu^{-1}C_\nu$ belongs to
$S(T_\nu^\sigma)=S(T_\nu)$.  
The entries $c_{rs}$ of $B_\nu$ are elements of $S(T_\nu)$.
This proves
\eqref{eq:block-matrix-form} and the necessity of
\eqref{eq:equivalence-data}.

Conversely, suppose that \eqref{eq:equivalence-data} holds and choose
matrices as in \eqref{eq:block-matrix-form}.  Both $B_\nu$ and its
inverse have entries in $S(T_\nu)$, so
\[
 B_\nu(T_\nu^\sigma)^{m_\nu}
 =(T_\nu^\sigma)^{m_\nu}.
\]
Consequently,
\[
 a_\nu B_\nu(T_\nu^\sigma)^{m_\nu}
 =(a_\nu T_\nu^\sigma)^{m_\nu}
 =(T'_{\pi(\nu)})^{m'_{\pi(\nu)}}.
\]
Combining these invertible block maps with the permutation $\pi$
defines $A\in\GL(k,q^n)$ satisfying $A(U^\sigma)=U'$.  The resulting
semilinear map uses the same companion automorphism $\sigma$ on
every block.
\end{proof}

The linear stabilizer is obtained by taking $U'=U$ and
$\sigma=\mathrm{id}$.

\begin{corollary}\label{cor:linear-stabilizer}
Let $U=T_1^{m_1}\times\cdots\times T_c^{m_c}$ be an admissible block
product, and write
$S(T_\nu)=\F_{q^{j_\nu}}$.  Then
\[
 \Stab_{\GL(k,q^n)}(U)
 \cong\prod_{\nu=1}^c\GL(m_\nu,q^{j_\nu}),
\]
where the factors are embedded block diagonally.
\end{corollary}

\begin{proof}
By \cref{thm:semilinear-equivalence}, since the representatives are
pairwise scalar-inequivalent, every block is fixed.  On the $\nu$th
block, the theorem gives a matrix $a_\nu B_\nu$ with
$T_\nu=a_\nu T_\nu$ and $B_\nu\in\GL(m_\nu,S(T_\nu))$.
Thus $a_\nu\in S(T_\nu)^*$, and the block matrix belongs to
$\GL(m_\nu,S(T_\nu))$.  Conversely, every such block-diagonal matrix
preserves $U$.
\end{proof}

\begin{remark}
These results describe the vector equivalence between admissible block products and calculate the stabilizers that preserve them. Each equivalence induces an equivalence of the associated weighted configurations $(L_U,w_U)$. The converse is neither asserted nor used here; the matrix-coefficient inclusions in the proof of \cref{thm:semilinear-equivalence} rely on $\Phi(U)=U'$ rather than merely on the preservation of point weights. The description of the projective equivalence between two (weighted or unweighted) linear sets associated with admissible block data is left as an open problem.
\end{remark}

\section{\texorpdfstring{Non-representable rank-two $q$-matroid interpretation}{Non-representable rank-two q-matroid interpretation}}\label{sec:q_mat}

The point-weight distribution of a linear set admits a natural interpretation in terms of the rank-one flats of a representable $q$-matroid, providing a direct bridge between the geometry of $(r,i)$-regular fat linear sets and $q$-matroid representability.

\begin{definition}
Let $E$ be a finite-dimensional $\F_q$-vector space and denote by
$\mathcal L(E)$ its lattice of $\F_q$-subspaces.  Recall that a \emph{$q$-matroid}
on $E$ is a pair $\mathcal M=(E,\rho)$, where $\rho:\mathcal L(E)\longrightarrow \mathbb Z_{\geq 0}$ is a \emph{rank function} satisfying, for all $A,B\leq_{\fq} E$, the following three properties, called \emph{ boundedness}, \emph{monotonicity}, and \emph{submodularity}, respectively:
\begin{enumerate}[label=\textup{(R\arabic*)}]
\item $0\leq \rho(A)\leq \dim_{\F_q}A$;
\item if $A\leq_{\fq} B$, then $\rho(A)\leq\rho(B)$;
\item $\rho(A+B)+\rho(A\cap B)\leq \rho(A)+\rho(B)$.

\end{enumerate}
The \emph{rank of $\mathcal M$} is $\rho(E)$.  
We say that $\mathcal M$ is \emph{loopless} if every one-dimensional subspace of $E$ has rank one.
\end{definition}

Recall that a subspace $F\leq_{\fq} E$ is a \emph{flat} of the $q$-matroid $\mathcal M=(E,\rho)$ if $\rho(F+\langle x\rangle_{\F_q})>\rho(F)$ for every $x\in E\setminus F$. Equivalently, $F$ is maximal, with respect to inclusion, among the
subspaces having rank $\rho(F)$.  In particular, a rank-one flat is
precisely a maximal subspace of rank one. 
We next recall the notion of cyclicity.  A subspace $C\leq_{\fq} E$ is \emph{cyclic} if $\rho(D)=\rho(C)$ for every $\F_q$-hyperplane $D$ of $C$.  A subspace which is both
cyclic and a flat is called a 
\emph{cyclic flat}.  We denote by
$\mathcal Z(\mathcal M)$ the set of cyclic flats of $\mathcal M$. We refer to \cite{JP18,AB24} for the basic theory of flats and cyclic flats of $q$-matroids.

Let $m,n$ be positive integers\footnote{In the context of rank-metric code theory, it is customary to frame the following discussion with reference to the field of cardinality $q^m$. However, to maintain consistency with the content of this paper, we will use the field $\fqn$.}. An \emph{$[m,k]_{q^n/q}$-system} (or more concisely \emph{$q$-system}) is an $m$-dimensional $\F_q$-subspace $U\leq_{\fq} \F_{q^n}^{\,k}$ such that $\langle U\rangle_{\F_{q^n}}=\F_{q^n}^{\,k}$.
For $X\leq_{\F_q}U$, put
\[
\rho_{U}(X)
=
\dim_{\F_{q^n}}
\langle X\rangle_{\F_{q^n}}.
\]
A $q$-matroid $\mathcal M=(E,\rho)$ of rank $k$ is
 \emph{$\F_{q^n}$-representable} if there exist an
$[m,k]_{q^n/q}$-system $U$ and an $\F_q$-linear isomorphism $\psi:E\longrightarrow U$, such that
\begin{equation}
\label{eq:qmat-representation}
\rho(X)=\rho_{U}(\psi(X))
\qquad
(X\leq_{\fq} E).
\end{equation}
We simply say that $\mathcal M$ is  {representable} if it is
$\F_{q^n}$-representable for some $n$.  This is equivalent to the usual
definition in terms of a generator matrix of an $\F_{q^n}$-linear
rank-metric code; see \cite{JP18,AB24,AJNZ26}.

Recall that a point $P\in\PG(k-1,q^n)$, identified with its corresponding
one-dimensional $\F_{q^n}$-subspace of $\F_{q^n}^k$, has weight $w_{ U}(P)=\dim_{\F_q}( U\cap P)$. We record the following standard observation relating the rank-one flats of a represented \(q\)-matroid to the point weights of its representing \(q\)-system; see, e.g., \cite{AB24,ABF25} for the \(q\)-system and rank-two geometric viewpoint.

\begin{proposition}
\label{prop:qmatroid-dictionary}
Let $\mathcal M=(E,\rho)$ be a loopless $q$-matroid of rank two and suppose
that $U\leq\F_{q^n}^2$ is a $q$-system representing $\mathcal M$ via the isomorphism $\psi$.
For $P\in\PG(1,q^n)$, set $F_P=\psi^{-1}(U\cap P)$.
Then the nonzero spaces $F_P$ are precisely the rank-one flats of
$\mathcal M$, that is, the maximal subspaces of rank one.  Moreover,
\begin{equation}
\label{eq:qmatroid-point-weight}
\dim_{\F_q}F_P=w_{U}(P).
\end{equation}
\end{proposition}
\begin{proof}
If $F_P\neq 0$, then
$\psi(F_P)\subseteq P$,
and therefore $\rho(F_P)=1$.  

Conversely, if $0\neq X\leq_{\fq} E$ has rank
one, then
\[
\dim_{\F_{q^n}}
\langle\psi(X)\rangle_{\F_{q^n}}=1.
\]
Hence $\psi(X)$ is contained in a unique point $P$ of
$\PG(1,q^n)$, and consequently $X\leq_{\fq} F_P$. 
Thus, the nonzero $\fq$-subspaces of type $F_P$ are exactly the maximal rank-one subspaces; hence, the  rank-one flats.

Finally, $\psi$ restricts to an $\F_q$-linear isomorphism
$F_P\longrightarrow U\cap P$,
and \eqref{eq:qmatroid-point-weight} follows.
\end{proof}

Hence, for a rank-two represented $q$-matroid, the point-weight
distribution of its representing $q$-system is exactly the dimension
distribution of its rank-one flats.

Let
$E=\F_q^{\,2i}$, $i\geq2$,
and let
$\mathcal V=\{V_1,\ldots,V_s\}$
be an partial $i$-spread of $E$, that is,
\[
\dim_{\F_q}V_a=i,
\qquad
V_a\cap V_b=0
\quad(a\neq b).
\]
Complete $\mathcal V$ to a vector-space partition by adjoining all
one-dimensional subspaces that are not contained in one of the
$V_a$.  More precisely, put
\[
\mathcal P(\mathcal V)
=
\mathcal V
\cup
\left\{
L\leq_{\fq} E:
\dim_{\F_q}L=1,\,
L\nleq_{\fq} V_a\text{ for every }a
\right\}.
\]
Define
\begin{equation}
\label{eq:partial-spread-rank}
\rho_{\mathcal V}(X)=
\begin{cases}
0,&X=0,\\[1mm]
1,&0\neq X\leq_{\fq} B
   \text{ for some }B\in\mathcal P(\mathcal V),\\[1mm]
2,&\text{otherwise}.
\end{cases}
\end{equation}

\begin{lemma}
\label{lem:partial-spread-qmatroid}
The function $\rho_{\mathcal V}$ in
\eqref{eq:partial-spread-rank} is the rank function of a loopless
rank-two $q$-matroid, which we denote by
$\mathcal M_{\mathcal V}=(E,\rho_{\mathcal V})$.
\end{lemma}

\begin{proof}
We verify the rank axioms.

Boundedness follows from the definition.  Indeed, if
$\rho_{\mathcal V}(X)=2$, then $X$ cannot be one-dimensional, since
every one-dimensional subspace belongs to a unique block of
$\mathcal P(\mathcal V)$.  Hence $\dim X\geq2$.

Monotonicity is immediate.  In fact, if $\rho_{\mathcal V}(X)=2$ and
$X\leq_{\fq} Y$, then $X$ is contained in no partition block, and therefore
neither is $Y$.

It remains to verify submodularity.  Cases involving the zero space
are immediate.  Suppose first that
$\rho_{\mathcal V}(X)=\rho_{\mathcal V}(Y)=1$.
Choose partition blocks $B_X,B_Y$ containing $X$ and $Y$, respectively.
If $B_X=B_Y$, then $X+Y\leq_{\fq} B_X$, and hence
\[
\rho_{\mathcal V}(X+Y)+
\rho_{\mathcal V}(X\cap Y)\leq2.
\]
If $B_X\neq B_Y$, then $B_X\cap B_Y=0$, so $X\cap Y=0$, while
$\rho_{\mathcal V}(X+Y)\leq2$.  Again the required inequality follows.

If $\rho_{\mathcal V}(X)=2$ and $\rho_{\mathcal V}(Y)=1$, then
$\rho_{\mathcal V}(X+Y)=2$ and
$\rho_{\mathcal V}(X\cap Y)\leq1$, giving
\[
\rho_{\mathcal V}(X+Y)+
\rho_{\mathcal V}(X\cap Y)\leq3
=
\rho_{\mathcal V}(X)+\rho_{\mathcal V}(Y).
\]
Finally, if both $X$ and $Y$ have rank two, then
\[
\rho_{\mathcal V}(X+Y)+
\rho_{\mathcal V}(X\cap Y)
\leq4
=
\rho_{\mathcal V}(X)+\rho_{\mathcal V}(Y).
\]
Thus $\rho_{\mathcal V}$ satisfies submodularity.  Since every
one-dimensional subspace has rank one, the resulting $q$-matroid is
loopless.
\end{proof}

We can determine the cyclic flats of  $\mathcal M_{\mathcal V}$.

\begin{proposition}
\label{prop:cyclic-flats-partial-spread}
The cyclic flats of $\mathcal M_{\mathcal V}$ are precisely
\begin{equation}
\label{eq:cyclic-flats-partial-spread}
\mathcal Z(\mathcal M_{\mathcal V})
=
\{0,V_1,\ldots,V_s,E\}.
\end{equation}
Their ranks are
\[
\rho_{\mathcal V}(0)=0,
\qquad
\rho_{\mathcal V}(V_a)=1
\quad(a=1,\ldots,s),
\qquad
\rho_{\mathcal V}(E)=2.
\]
\end{proposition}

\begin{proof}
The rank-one flats are exactly the blocks of
$\mathcal P(\mathcal V)$.  Indeed, if $B$ is such a block and
$x\nleq B$ is a one-dimensional subspace, then $B+x$ is contained in
no partition block and therefore has rank two.

Each $V_a$ is cyclic.  Since $i\geq2$, every hyperplane
$D<V_a$ is nonzero and satisfies
\[
\rho_{\mathcal V}(D)=1=\rho_{\mathcal V}(V_a).
\]
On the other hand, a one-dimensional block
$L\in\mathcal P(\mathcal V)\setminus\mathcal V$ is not cyclic, since
its unique hyperplane is $0$ and
\[
\rho_{\mathcal V}(0)=0<1=\rho_{\mathcal V}(L).
\]

The whole space $E$ is cyclic as well.  Indeed, every hyperplane
$H<E$ has dimension $2i-1>i$, and hence cannot be contained in a
partition block.  Consequently,
\[
\rho_{\mathcal V}(H)=2=\rho_{\mathcal V}(E).
\]

Finally, the only rank-zero flat is $0$, and the only rank-two flat is
$E$: a proper subspace of rank two can always be enlarged without
increasing its rank.  This proves \eqref{eq:cyclic-flats-partial-spread}.
\end{proof}

The next result establishes a connection between the representability of a $q$-matroid and the structure of  a regulus in a projective space in some cases. In a projective space of dimension $2i-1$, a regulus of $(i-1)$-dimensional subspaces is a set of pairwise skew $(i-1)$-subspaces, whose union is a Segre variety $S_{1,i-1}$. See \cite{BB64,Seg64} for more information on reguli of lines and \cite{LZ14} for a geometric characterization of the Segre varieties.

%definition of a regulus of $(i-1)$-dimensional subspaces.
\begin{definition}
    A \emph{regulus of $(i-1)$-dimensional subspaces} in a projective space $\PG(n,q)$ is a set $\mathcal R$ of $q+1$ pairwise skew $(i-1)$-dimensional subspaces, such that for every point $P$ on a subspace of $\mathcal R$, there is a line that contains $P$ and intersects all elements of $\mathcal R$.
\end{definition}
In the case $i=2$, we obtain the standard definition of a regulus of lines in a three-dimensional finite projective space.

\begin{theorem}
\label{thm:partial-spread-rigidity}
Let $i\geq2$, let
$
\mathcal V=\{V_1,\ldots,V_s\}
$
be an partial $i$-spread of $E=\F_q^{2i}$, and assume
$
s<q^i+1$.
If $\mathcal M_{\mathcal V}$ is representable over a finite extension
of $\F_q$, then
\begin{equation}
\label{eq:qmatroid-s-values}
s\leq2
\qquad\text{or}\qquad
s=q+1.
\end{equation}

Moreover, 
if $s=q+1$, then there is an $\fq$-linear isomorphism $\varphi:V_1\longrightarrow V_s$ for which
\begin{equation}\label{eq:qmatroid-new}
\mathcal V=\set{W_\lambda:\lambda \in\fq}\cup\set{V_s},\quad\text{where }\ W_\lambda=\set{x+\lambda\varphi(x):x\in V_1}.
\end{equation}
Furthermore, $\mathcal V$ defines a regulus of $(i-1)$-dimensional subspaces in $\PG(2i-1,q)$.
\end{theorem}

\begin{proof}
Assume that $\mathcal M_{\mathcal V}$ is
$\F_{q^n}$-representable, and let $U\leq\F_{q^n}^{\,2}$ be a representing $q$-system.  Thus there is an $\F_q$-linear
isomorphism
$\psi:E\longrightarrow U$
such that
\[
\rho_{\mathcal V}(X)
=
\dim_{\F_{q^n}}
\langle\psi(X)\rangle_{\F_{q^n}}
\qquad
(X\leq_{\fq} E).
\]
By \cref{prop:qmatroid-dictionary}, 
the nonzero subspaces of type $F_P=\psi^{-1}(U\cap P)$ are the rank-one flats
of $\mathcal M_{\mathcal V}$.
By construction, the latter are exactly
the blocks of $\mathcal P(\mathcal V)$.  Consequently, the spaces
$V_1,\ldots,V_s$ correspond to $s$ projective points of weight $i$,
whereas every other point of $L_{U}$ has weight one.

Since
$
s<q^i+1$,
the spaces $V_1,\ldots,V_s$ do not cover all nonzero vectors of $E$:
indeed,
$
s(q^i-1)<q^{2i}-1$.
Hence $\mathcal P(\mathcal V)$ contains at least one
one-dimensional block.  It follows that $L_{U}$ contains a
point of weight one.  Therefore $L_{U}$ is an
$(s,i)$-regular fat linear set of rank $2i$ in
$\PG(1,q^n)$.

Applying \cref{thm:line-rigidity} gives $s\leq2$ or $s=q+1$, which proves \eqref{eq:qmatroid-s-values}.

 Now assume $s=q+1$.
 A projectivity exists mapping any three distinct points of weight $i$ to the standard frame of $\PG(1,q^n)$.
As in \cite[Proposition~6.1]{SZZ26} and the proof therein, and \cref{prop:qplusone}, we can require, for a related $\fqn$-linear isomorphism $\kappa_0\in\GL(2,q^n)$, that
\[
\kappa_0(\psi(V_1))=T\times 0,\quad \kappa_0(\psi(V_s))=0\times T,\quad\kappa_0(U)=T\times T
\]
for a suitable $i$-dimensional $\fq$-subspace $T$ of $\fqn$.
Since $\kappa_0$ preserves weights,  also $T\times T$  is a representing $q$-system via
\[
\kappa=\kappa_0\circ\psi:E\longrightarrow T\times T.
\]
Indeed, for every $X\le_{\fq}E$ we have
\[
\rho_\mathcal V(X)=\rho_U(\psi(X))=\dim_{\fqn}\angles{\psi(X)}_{\fqn}=\dim_{\fqn}\angles{\kappa(X)}_{\fqn}=\rho_{T\times T}\left(\kappa(X)\right).
\]
Define $\sigma(x,y)=(y,x)$.
We have
\[
\varphi=\kappa^{-1}\circ\sigma\circ\kappa\in\GL(E,q),
\]
and $\varphi(V_1)=V_s$.
By \cref{prop:qplusone}, the fat points in $L_{T\times T}$, distinct from $\angles{(0,1)}_{\fqn}$, are associated with the following subspaces of $T\times T$:
\[
\set{y+\lambda\sigma(y):y\in\kappa(V_1)},\quad\lambda\in \fq.
\]
Now, to obtain \cref{eq:qmatroid-new}, take the preimage through $\kappa$ of such spaces, and put $y=\kappa(x)$.

Finally, for every $0\neq x\in V_1$, the projective line
\[
\ell_x=\PG\bigl(\langle x,\varphi(x)\rangle_{\fq}\bigr)
\]
meets each $\PG(W_\lambda)$ at $\langle x+\lambda\varphi(x)\rangle_{\fq}$
and meets $\PG(V_s)$ at $\langle\varphi(x)\rangle_{\fq}$.
Every point of the subspaces defined by $\mathcal V$ lies on one
of these lines. Since these subspaces are pairwise disjoint,
they form a regulus by definition.
\end{proof}

We immediately obtain an infinite family of nonrepresentable $q$-matroids.

\begin{corollary}
\label{cor:qmatroid-nonrep}
For every $i\geq2$ and every integer $s$ satisfying
\[
3\leq s<q^i+1,
\qquad
(q,s)\neq (2,3),
\]
there exists a loopless rank-two $q$-matroid on $\F_q^{2i}$ that is
not representable over any finite extension of $\F_q$ having exactly $s$ rank-one cyclic flats.
\end{corollary}
\begin{proof}
Identify $E$ with $K^2$, where $K=\F_{q^i}$. For each $a\in K$, put
\[
V_a=\{(x,ax):x\in K\},
\qquad
V_\infty=\{(0,x):x\in K\}.
\]
These subspaces form an $i$-spread
$\mathcal S=\{V_a:a\in K\}\cup\{V_\infty\}$ of $E$.
If $2<s<q^2+1$ and $s\neq q+1$, choose any $s$ members of $\mathcal S$ and
denote the resulting partial spread by $\mathcal V$.
By \cref{thm:partial-spread-rigidity}, the associated
$q$-matroid $\mathcal M_{\mathcal V}$ is not representable.
It remains to consider $s=q+1$. The excluded case $(q,s)=(2,3)$
allows us to assume $q>2$. Choose $\alpha\in K\setminus\F_q$ and set
\[
\mathcal V=\{V_a:a\in\F_q\}\cup\{V_\alpha\}.
\]
Thus $\mathcal V$ is obtained from the standard regulus by replacing
$V_\infty$ with $V_\alpha$. We show that $\mathcal V$ does not
define a regulus.
Choose $\mu\in\F_q\setminus\{0,1\}$. A projective line meeting
$\PG(V_0)$ and $\PG(V_1)$ has underlying vector space
$L=\langle(x,0),(y,y)\rangle_{\F_q}$
for some $x,y\in K^*$.
If it also meets $\PG(V_\mu)$, there exist $a,b\in\F_q^*$ such that
\[
by=\mu(ax+by).
\]
Hence $x$ and $y$ are proportional over $\F_q$, and consequently
$L=\langle(x,0),(0,x)\rangle_{\F_q}$.
Such a line cannot meet $\PG(V_\alpha)$, since a nonzero vector
$(cx,dx)\in V_\alpha$, with $c,d\in\F_q$, would force
$c\neq0$ and $\alpha=d/c\in\F_q$.
Therefore no line meets every member of $\mathcal V$, so
$\mathcal V$ does not define a regulus.
By \cref{thm:partial-spread-rigidity},
$\mathcal M_{\mathcal V}$ is again not representable.
In both cases, \cref{lem:partial-spread-qmatroid,prop:cyclic-flats-partial-spread}
show that $\mathcal M_{\mathcal V}$ is loopless of rank two
and has exactly $s$ rank-one cyclic flats.
\end{proof}

\begin{remark}
\label{rem:known-nonrep-qmatroids}
Nonrepresentable $q$-matroids have been previously constructed in
\cite{CJ24,GLJ22,GLJ25}; moreover, Degen and K\"uhne \cite{DK26} proved that
asymptotically almost all $q$-matroids are nonrepresentable.
In particular, \cite{CJ24} exhibits the smallest nonrepresentable
$q$-matroid, a rank-two $q$-matroid on $\F_2^4$ with four
two-dimensional circuits.  This example is recovered from our
construction by taking $q=2$, $i=2$, and $s=4$.
Thus \cref{thm:partial-spread-rigidity,cor:qmatroid-nonrep} provide
a uniform representability obstruction for this class of rank-two
$q$-matroids over arbitrary $q$ and $i$, in terms of the geometry
and number of their rank-one cyclic flats.
\end{remark}

\section*{Declaration of Generative AI and AI-assisted Technologies in the Writing Process}
During the preparation of this work, the authors utilized generative AI tools solely for the initial drafting phase. The authors reviewed, substantially revised, and finalized all content, ensuring rigorous oversight of all mathematical derivations, literature citations, and intellectual claims. The authors assume full responsibility for the contents of the publication.

%\section*{Funding}
%{}

\section*{Acknowledgements}
The authors would like to thank Gianira N. Alfarano for her useful comments and suggestions. 

\medskip

The research was partially supported by the Italian National Group for Algebraic and Geometric Structures and their Applications (GNSAGA - INdAM). The research of the second and the fourth author was supported by the project SID \emph{Linear Algebra methods in Combinatorics and Industrial Engineering: Coding Theory, Graphs and their interest in Assembly Line Problems} CUP C33C25001110005 of Dipartimento di Tecnica e Gestione dei Sistemi Industriali of the Università degli Studi di Padova.


\begin{thebibliography}{AJNZ26}

\bibitem[AB24]{AB24}
G.~N. Alfarano and E. Byrne,
\emph{The cyclic flats of a $q$-matroid},
J. Algebraic Combin. \textbf{60} (2024), 97--126.

\bibitem[ABF25]{ABF25}
G.~N. Alfarano, E. Byrne, and A. Fulcher,
\emph{The free product of $q$-matroids},
Electron. J. Combin. \textbf{32} (2025), no.~3, Paper P3.52.

\bibitem[AJNZ26]{AJNZ26}
G.~N. Alfarano, R. Jurrius, A. Neri, and F. Zullo,
\emph{Representability of the direct sum of uniform $q$-matroids},
Combinatorial Theory \textbf{6} (2026), no.~1, Paper 4.

\bibitem[BL00]{BL00}
A. Blokhuis and M. Lavrauw,
\emph{Scattered spaces with respect to a spread in $\PG(n,q)$},
Geom. Dedicata \textbf{81} (2000), 231--243.

\bibitem[BB64]{BB64}
R.~H. Bruck and R.~C. Bose,
\emph{The construction of translation planes from projective spaces},
J. Algebra \textbf{1} (1964), 85--102.

\bibitem[BSZ17]{BSZ17}
C. Bachoc, O. Serra, and G. Z\'{e}mor,
\emph{An analogue of Sidon sets for subspaces of finite fields},
Math. Proc. Cambridge Philos. Soc. \textbf{163} (2017), 225--239.

\bibitem[CJ24]{CJ24}
M. Ceria and R. Jurrius,
\emph{The direct sum of $q$-matroids},
J. Algebraic Combin. \textbf{59} (2024), 291--330.

\bibitem[CMP24]{CMP24}
B. Csajb\'{o}k, G. Marino, and V. Pepe,
\emph{On the maximum field of linearity of linear sets},
Bull. Lond. Math. Soc. \textbf{56} (2024), 3300--3315.

\bibitem[DK26]{DK26}
S.~Degen and L.~K\"uhne,
\emph{Most $q$-matroids are not representable},
to appear in Combinatorial Theory;
arXiv:2408.06795v2, 2025.

\bibitem[FS06]{FS06}
S.~Z.~Fancsali and P.~Sziklai,
\emph{About maximal partial $2$-spreads in $\PG(3m-1,q)$},
Innov. Incidence Geom. \textbf{4} (2006), 89--102.

\bibitem[GLJ22]{GLJ22}
H. Gluesing-Luerssen and B. Jany,
\emph{$q$-polymatroids and their relation to rank-metric codes},
J. Algebraic Combin. \textbf{56} (2022), no.~3, 725--753.

\bibitem[GLJ25]{GLJ25}
H. Gluesing-Luerssen and B. Jany,
\emph{Representability of the direct sum of $q$-matroids},
J. Algebraic Combin. \textbf{61} (2025), no.~4, Paper~51.



\bibitem[JP18]{JP18}
R. Jurrius and R. Pellikaan,
\emph{Defining the $q$-analogue of a matroid},
Electron. J. Combin. \textbf{25} (2018), no.~3, Paper P3.2.

\bibitem[LZ14]{LZ14}
M.~Lavrauw and C.~Zanella,
\emph{Segre embeddings and finite semifields},
Finite Fields Appl. \textbf{25} (2014), 8--18.

\bibitem[LZ15]{LZ15}
M.~Lavrauw and C.~Zanella,
\emph{Subgeometries and linear sets on a projective line},
Finite Fields Appl. \textbf{34} (2015), 95--106.

\bibitem[LN97]{LN97}
R. Lidl and H. Niederreiter,
\emph{Finite Fields}, 2nd ed.,
Encyclopedia of Mathematics and its Applications, vol.~20,
Cambridge University Press, Cambridge, 1997.


\bibitem[Lun84]{Lun84}
G. Lunardon, \emph{Fibrazioni planari e sottovarietà algebriche della varietà di Grassmann}, Geom.\ Dedicata \textbf16 (1984), 291--313.

\bibitem[NPSZ22]{NPSZ22}
V.~Napolitano, O.~Polverino, P.~Santonastaso, and F.~Zullo,
\emph{Linear sets on the projective line with complementary weights},
Discrete Math. \textbf{345} (2022), no.~7, Paper 112890.

\bibitem[NZ26]{NZ26}
A.~Neri and F.~Zullo,
\emph{Rank-metric codes over arbitrary fields: Bounds and constructions},
Art Discrete Appl. Math. \textbf{9} (2026), no.~2, Paper P2.09.

\bibitem[Pol10]{Pol10}
O. Polverino,
\emph{Linear sets in finite projective spaces},
Discrete Math. \textbf{310} (2010), 3096--3107.

\bibitem[PZ20]{PZ20}
O.~Polverino and F.~Zullo,
\emph{Connections between scattered linear sets and MRD-codes},
Bull. Inst. Combin. Appl. \textbf{89} (2020), 46--74.

\bibitem[RRT18]{RRT18}
R.~M. Roth, N. Raviv, and I. Tamo,
\emph{Construction of Sidon spaces with applications to coding},
IEEE Trans. Inform. Theory \textbf{64} (2018), 4412--4422.

\bibitem[Seg64]{Seg64}
B. Segre,
\emph{Teoria di Galois, fibrazioni proiettive e geometrie non desarguesiane},
Ann. Math Pura Appl. \textbf{64} (1964), no. 4, 1--76.

\bibitem[SZZ26]{SZZ26}
V. Smaldore, C. Zanella, and F. Zullo,
\emph{Regular fat linear sets},
Des. Codes Cryptogr.,
\textbf{94} (2026), no.~9, Paper 191.

\bibitem[Zul23]{Zul23}
F. Zullo,
\emph{Multi-orbit cyclic subspace codes and linear sets},
Finite Fields Appl. \textbf{87} (2023), Paper 102153.

\end{thebibliography}
\end{document}